\documentclass[12pt, a4paper]{article}

\usepackage{tikz}
\usepackage{amsmath}
\usepackage{amsfonts}
\usepackage{amssymb}
\usepackage{amsthm}
\usepackage{enumerate}
\usepackage{geometry}
\title{An Asymptotic Version of the Multigraph 1-Factorization Conjecture}
\author{E.~R.~Vaughan \\
School of Mathematical Sciences \\
Queen Mary, University of London \\
e-mail: \texttt{e.vaughan@qmul.ac.uk}}

\date{\today}

\newtheorem{theorem}{Theorem}
\newtheorem{lemma}[theorem]{Lemma}
\newtheorem{corollary}[theorem]{Corollary}
\newtheorem{conjecture}[theorem]{Conjecture}

\DeclareMathOperator{\pr}{Pr}
\DeclareMathOperator{\ex}{E}
\providecommand{\size}[1]{\left|#1\right|}
\providecommand{\floor}[1]{\lfloor#1\rfloor}
\providecommand{\ceiling}[1]{\lceil#1\rceil}

\begin{document}

\maketitle

\begin{abstract} We give a self-contained proof that for all positive integers $r$ and all $\epsilon > 0$, there is an integer $N = N(r, \epsilon)$ such that for all $n \ge N$ any regular multigraph of order $2n$ with multiplicity at most $r$ and degree at least $(1+\epsilon)rn$ is 1-factorizable. This generalizes results of Perkovi{\'c} and Reed, and Plantholt and Tipnis. \end{abstract}

\section{Introduction}

In 1985 Chetwynd and Hilton \cite{ch85} made the following conjecture, which is often called the ``1-Factorization Conjecture'':

\begin{conjecture} Any regular simple graph of order $2n$ and degree at least $n$ is 1-factorizable. \label{1fact} \end{conjecture}

Should this conjecture be true, a pleasant consequence is that for any regular graph $G$ of even order, at least one of $G$ and its complement is 1-factorizable. A natural generalization to multigraphs of bounded multiplicity was made subsequently by Plantholt and Tipnis \cite{plantholttipnis2} (see also \cite{plantholttipnis1}):

\begin{conjecture} Let $G$ be a regular multigraph of order $2n$ with multiplicity at most $r$. If the degree of $G$ is at least $rn$ then $G$ is 1-factorizable. \label{m1fact} \end{conjecture}

If true, Conjecture \ref{m1fact} is best possible for every $r \ge 1$, at least when $n$ is odd. This is demonstrated by the following construction. Suppose $r$ and $n$ are positive integers where $n$ is odd and $r > 1$. Consider the graph $H$ of order $2n$, formed from three graphs $A$, $B$ and $M$. $A$ and $B$ are complete graphs on $n$ vertices each, and $M$ is a matching of $n$ edges, in which each edge joins a vertex in $A$ with a vertex in $B$. (See Figure \ref{bestposs}.) Let $G$ be the multigraph obtained from $H$ by replacing each edge of $M$ by $r - 1$ parallel edges, and each other edge by $r$ parallel edges.

As $A$ and $B$ each have an odd number of vertices, any 1-factor of $G$ must contain an edge that joins a vertex in $A$ with a vertex in $B$. There are only $n(r - 1)$ such edges, so there can be at most $n(r - 1)$ disjoint 1-factors. As $G$ has degree $rn - 1$, it is not 1-factorizable.

In the case where $r = 1$, and $n$ is odd, we can take $G$ to be the disjoint union of two complete graphs on $n$ vertices. $G$ is regular of degree $n - 1$, and is not 1-factorizable, as it has no 1-factors at all.

The following approximate resolution of Conjecture \ref{1fact} was obtained by H{\"a}ggkvist (unpublished) and independently by Perkovi{\'c} and Reed \cite{perkovicreed}:

\begin{theorem} For any $\epsilon > 0$ there is an integer $N = N(\epsilon)$ such that for all $n \ge N$ any regular simple graph of order $2n$ with degree at least $(1+\epsilon)n$ is 1-factorizable. \label{a1fact} \end{theorem}

In this note, we shall prove the following generalization of Theorem \ref{a1fact}, which is an approximate version of Conjecture \ref{m1fact}:

\begin{theorem} For all positive integers $r$ and all $\epsilon > 0$, there is an integer $N = N(r, \epsilon)$ such that for all $n \ge N$ any regular multigraph of order $2n$ with multiplicity at most $r$ and degree at least $(1+\epsilon)rn$ is 1-factorizable. \label{ma1fact} \end{theorem}

In previous work, Plantholt and Tipnis have obtained this result in the special case where $r$ is even \cite{plantholttipnis2}. They employed a method of factorizing a multigraph into simple graphs, to which they applied Theorem \ref{a1fact}. Our approach is different, and does not need to distinguish between even and odd $r$. Our proof of Theorem \ref{ma1fact} is based on Perkovi{\'c} and Reed's proof of Theorem \ref{a1fact}, although we have simplified the argument in a number of respects, and so the proof presented here is shorter and simpler.

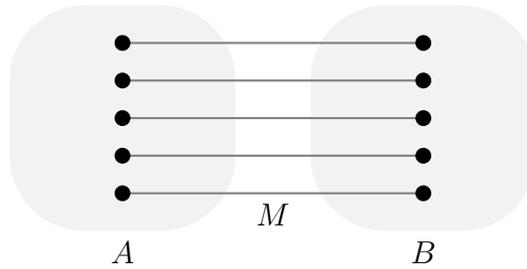
\begin{figure} \begin{center} \begin{tikzpicture}
\fill[gray!10,rounded corners=10mm] (0,0) rectangle +(3,3) (4,0) rectangle +(3,3);
\draw[gray, thick] (1.5,2.5) -- (5.5,2.5);
\draw[gray, thick] (1.5,2.0) -- (5.5,2.0);
\draw[gray, thick] (1.5,1.5) -- (5.5,1.5);
\draw[gray, thick] (1.5,1.0) -- (5.5,1.0);
\draw[gray, thick] (1.5,0.5) -- (5.5,0.5);
\fill[black] (1.5,2.5) circle (3pt) (5.5,2.5) circle (3pt);
\fill[black] (1.5,2.0) circle (3pt) (5.5,2.0) circle (3pt);
\fill[black] (1.5,1.5) circle (3pt) (5.5,1.5) circle (3pt);
\fill[black] (1.5,1.0) circle (3pt) (5.5,1.0) circle (3pt);
\fill[black] (1.5,0.5) circle (3pt) (5.5,0.5) circle (3pt);
\draw[black] (1.5,0) node[below] {$A$} (5.5,0) node[below] {$B$} (3.5,0.5) node[below] {$M$};
%
%
\end{tikzpicture} \end{center}
\caption{The graph $H$. $A$ and $B$ are complete graphs on $n$ vertices, and $M$ is a matching of $n$ edges.}
\label{bestposs} \end{figure}

\section{Preliminaries}

We shall begin by giving some definitions. We omit definitions of some of the most basic concepts in graph theory, which can be found, for example, in \cite{bondymurty}. Unless stated otherwise, all graphs will be multigraphs. By this we mean that they may contain multiple edges, but do not contain any loops. The vertex set and edge set of a graph $G$ are denoted $V(G)$ and $E(G)$ respectively. A set of edges is said to be \emph{parallel} if each edge joins the same pair of vertices. The \emph{multiplicity} of a graph $G$ is the maximum size of a set of parallel edges. A graph of multiplicity 1 is said to be \emph{simple}. The degree of a vertex $v \in V(G)$ is denoted $d(v)$. In the case that $G$ is regular, $d(G)$ denotes the degree of every vertex. The maximum and minimum degrees of $G$ are denoted $\Delta(G)$ and $\delta(G)$ respectively. Given a set of vertices $S$ and a vertex $v \in V(G)$, $d_S(v)$ is the number of edges of the form $vs$ where $s \in S$. The set of vertices that are adjacent to at least one vertex in $S$ is called the \emph{neighbour set} of $S$, denoted $N(S)$. The subgraph of $G$ induced by $S$ is denoted $G_S$.

A \emph{matching} in $G$ is a set of edges, no two of which are adjacent. Given a matching $M$, if a vertex $v \in V(G)$ is incident with an edge of $M$ then $v$ is said to be \emph{covered} by $M$, otherwise $v$ is \emph{missed} by $M$. A matching that covers every vertex is called a \emph{1-factor}. A \emph{1-factorization} of $G$ is a partition of $E(G)$ into disjoint 1-factors. A graph with a 1-factorization is said to be \emph{1-factorizable}. An \emph{edge-colouring} of $G$ is an assignment of colours to the edges of $G$ in which no two adjacent edges are given the same colour. The set of edges that are given a particular colour is called a \emph{colour class}. Since adjacent edges receive different colours, each colour class is a matching. The \emph{chromatic index} of $G$, denoted $\chi'(G)$, is the least number of colours needed for an edge-colouring. For a regular graph $G$, an edge-colouring with $d(G)$ colours is the same thing as a 1-factorization, as both are partitions of $E(G)$ into $d(G)$ disjoint matchings. In fact, in Section~\ref{keysection}, we shall show that a graph $G$ is 1-factorizable by giving a procedure for finding an edge-colouring of $G$ with $d(G)$ colours.

We shall need the following two classical theorems, both in their multigraph versions. (See e.g. \cite{bondymurty}.)

\begin{theorem} \emph{(Vizing's Theorem)} Let $G$ be a graph with multiplicity at most $r$. Then the chromatic index $\chi'(G)$ is at least $\Delta(G)$ and at most $\Delta(G)+r$. \label{vizing} \end{theorem}

\begin{theorem} \emph{(K\"onig's Theorem)} Let $G$ be a bipartite graph of any multiplicity. Then $\chi'(G) = \Delta(G)$. \label{konig} \end{theorem}

An edge-colouring of a graph $G$ with $k$ colours is said to be \textit{equalized} if each colour class contains either $\floor{\size{E(G)}/k}$ or $\ceiling{\size{E(G)}/k}$ edges. The following was first observed by McDiarmid \cite{mcdiarmid}:

\begin{theorem} Let $G$ be a graph of any multiplicity with chromatic index $\chi'(G)$. Then for all $k \ge \chi'(G)$ there is an equalized edge-colouring of $G$ with $k$ colours. \label{mcdiarmid} \end{theorem}

We shall also need Hall's Theorem \cite{bondymurty}:

\begin{theorem} Let $G$ be a bipartite graph of any multiplicity, with bipartition $(X,Y)$. There is a matching covering every vertex of $X$ if and only if $\size{N(S)} \ge \size{S}$ for all $S \subseteq X$. \label{hall} \end{theorem}

A standard consequence of Theorem \ref{hall} is the following:

\begin{lemma} Let $G$ be a bipartite simple graph, with bipartition $(X,Y)$, where $\size{X} = \size{Y} = n$. If $\delta(G) \ge n/2$ then $G$ has a 1-factor. \label{hallc} \end{lemma}

\begin{proof} Suppose $G$ satisfies the assumptions but does not have a 1-factor. Then by Theorem \ref{hall} there is a set $X' \subseteq X$ with neighbour set $Y' \subseteq Y$ such that $\size{X'} > \size{Y'}$. But $\delta(G) \ge n/2$, so $\size{Y'} \ge n/2$, and so $\size{X'} > n/2$. But then any vertex in $Y-Y'$ must be adjacent to at least one vertex in $X'$, which contradicts $Y'$ being the neighbour set of $X'$. \end{proof}

We can extend this result to bipartite multigraphs as follows:

\begin{lemma} Let $G$ be a bipartite graph of multiplicity at most $r$, with bipartition $(X,Y)$, where $\size{X} = \size{Y}$ = n. If $\delta(G) \ge rn/2$ then $G$ has a 1-factor. \label{hallcm} \end{lemma}

\begin{proof} Let $G'$ be the simple graph obtained from $G$ by replacing sets of parallel edges with single edges. As $G$ has multiplicity at most $r$, $\delta(G') \ge n/2$, and so $G'$ has a 1-factor by Lemma \ref{hallc}. Since any 1-factor of $G'$ is also a 1-factor of $G$, the result follows. \end{proof}

We shall need the following version of the Chernoff bound. (See e.g. Theorem A.1.16 of \cite{alonspencer}.)

\begin{theorem} Let $X_1, \dots, X_n$ be mutually independent random variables that satisfy $\ex(X_i) = 0$ and $\size{X_i} \le 1$ for all $1 \le i \le n$. Set $S = X_1 + \dots + X_n$. Then for any $a > 0$,
\[ \pr(S > a) < e^{-a^2/2n}. \]
\label{chernoff} \end{theorem}

Applying Theorem \ref{chernoff} to $S$ and $-S$ we obtain:

\begin{corollary} Let $X_1, \dots, X_n$ be as in Theorem \ref{chernoff}. Then,
\[ \pr(\size{S} > a) < 2e^{-a^2/2n}.\] \label{chernoff2}
\end{corollary}

\section{Proof of Theorem \ref{ma1fact}}

The following lemma states that for any fixed $r$, when $n$ is sufficiently large, the vertices of a graph $G$ of order $2n$ with multiplicity at most $r$ can be partitioned into two parts $A$ and $B$ such that for each vertex $v$, $d_A(v)$ and $d_B(v)$ are approximately equal.

\begin{lemma} For all positive integers $r$, there is an integer $N^* = N^*(r)$ such that for all $n \ge N^*$ the vertex set of any graph $G$ of order $2n$ with multiplicity at most $r$ can be partitioned into two equal parts $A$ and $B$ such that for any vertex $v$ we have
\begin{equation}
\size{d_A(v) - d_B(v)} < n^{2/3}. \label{splitseq}
\end{equation}
\label{splits} \end{lemma}

First, two remarks:

\begin{enumerate}
\item It is possible to replace $n^{2/3}$ with $\sqrt{n \log n}$.
\item The case $r = 1$ follows from a hypergeometric version of the Chernoff bound given by Chv\'atal \cite{chvatal}.
\end{enumerate}

\begin{proof}[Proof of Lemma \ref{splits}] Let $G$ be a graph of order $2n$ with multiplicity at most $r$. We shall show that provided $n$ is large enough, there is a method for randomly choosing a partition of $V(G)$ into two equal parts $A$ and $B$, such that with positive probability, (\ref{splitseq}) holds for every $v \in V(G)$.

Suppose we have partitioned $V(G)$ into $n$ pairs in an arbitrary way. We then assign one vertex of each pair to $A$ and the other to $B$ uniformly at random. Suppose the pairs are $(a_1, b_1), \dots, (a_n, b_n)$. Fix a vertex $v$, and define the random variables $X_1, \dots, X_n$ by the rule that
\[ X_i = \frac{m(v a_i) - m(v b_i)}{r}, \]
where $m(vx)$ denotes the number of edges between $v$ and $x$. Then $X_1, \dots, X_n$ are mutually independent, and for all $1 \le i \le n$, $\ex(X_i) = 0$ and $\size{X_i} \le 1$. Let $S = X_1 + \dots + X_n$. Then
\[ d_A(v) - d_B(v) = r S.\]
By Corollary \ref{chernoff2},
\begin{align*}
\pr(\size{d_A(v) - d_B(v)} > n^{2/3}) &= \pr(\size{S} > r^{-1}n^{2/3}) \\
 &< 2e^{-\tfrac{1}{2n}(r^{-1}n^{2/3})^2} \\
&= 2e^{-\tfrac{1}{2}r^{-2}n^{1/3}}.
\end{align*}
There are $2n$ vertices, so the probability $p$ that there is a vertex $v$ for which (\ref{splitseq}) does not hold is less than
\[ 4ne^{-\tfrac{1}{2}r^{-2}n^{1/3}}, \]
which tends to 0 as $n \rightarrow \infty$. Hence if $n$ is large enough, we can be certain that $p < 1$, and so there must be some partition of $V(G)$ into two equal parts $A$ and $B$ such that (\ref{splitseq}) holds for every $v \in V(G)$. \end{proof}

The proof of the following lemma will be deferred until Section \ref{keysection}:

\begin{lemma} Let $G$ be a regular graph of order $2n$ with multiplicity at most $r$, where $n^{5/6} > 3r$. If the vertex set can be partitioned into two equal parts $A$ and $B$ such that every vertex $v$ has $d_A(v) > rn/2 + 14rn^{5/6}$ and $d_B(v) > rn/2 + 14rn^{5/6}$, and where
\[ \max\{\Delta(G_A), \Delta(G_B)\} - \min\{\delta(G_A), \delta(G_B)\} < n^{2/3}, \]
then $G$ is 1-factorizable. \label{keylemma}
\end{lemma}

Note that Lemma \ref{keylemma} is a purely deterministic result, which applies to every graph satisfying the conditions. Indeed, our proof gives a deterministic algorithm for finding a 1-factorization of such a graph.

\begin{proof}[Proof of Theorem \ref{ma1fact}] Let $r$ be a positive integer and $\epsilon > 0$. Let $n$ be large enough so that $n^{5/6} > 3r$, $n \ge N^*(r)$ of Lemma \ref{splits} and
\begin{equation} (1 + \epsilon)rn > (1 + 29n^{-1/6})rn = rn + 29rn^{5/6}. \label{epsiloncond} \end{equation}
Suppose $G$ is a regular graph of order $2n$ with multiplicity at most $r$ and degree at least $(1 + \epsilon)rn$. By Lemma \ref{splits} we can partition the vertex set of $G$ into two equal parts $A$ and $B$ such that for every vertex $v$ we have
\[ \size{d_A(v) - d_B(v)} < n^{2/3}. \]
Since $G$ is regular of degree $d = d(G)$, and for every vertex $v$, $d_A(v) + d_B(v) = d$, we have
\[ \frac{d - n^{2/3}}{2} < d_A(v) < \frac{d + n^{2/3}}{2} \]
and so
\[ \frac{d - n^{2/3}}{2} < \delta(G_A) \le \Delta(G_A) < \frac{d + n^{2/3}}{2}. \]
Since the same is true of $G_B$, we have
\[ \max\{\Delta(G_A), \Delta(G_B)\} - \min\{\delta(G_A), \delta(G_B)\} < n^{2/3}. \]
By (\ref{epsiloncond}), $d(G) > rn + 28rn^{5/6} + n^{2/3}$, and so for every vertex $v$, we have
\[ d_A(v) > rn/2 + 14rn^{5/6} \enspace \text{and} \enspace d_B(v) > rn/2 + 14rn^{5/6}. \]
Thus $G$ is 1-factorizable by Lemma \ref{keylemma}.
\end{proof}

\section{Proof of Lemma \ref{keylemma}} \label{keysection}

In the course of the proof of Lemma \ref{keylemma} we shall be considering graphs where some of the edges are coloured and some are not. A path whose edges alternate between uncoloured edges and edges coloured $c$, for some colour $c$, will be called an \textit{alternating path}. To \textit{exchange} an alternating path $P$ means to uncolour the edges of $P$ that were previously coloured $c$, and to colour with $c$ the edges of $P$ that were previously uncoloured.

\begin{proof}[Proof of Lemma \ref{keylemma}] Suppose we have a regular graph $G$ of order $2n$ and a partition of its vertex set into two equal parts $A$ and $B$ such that the conditions in the statement of the lemma are satisfied. The subgraphs of $G$ induced by $A$ and $B$ will be denoted $G_A$ and $G_B$. Let $C$ be the subgraph of $G$ consisting of the edges that are not in $G_A$ or $G_B$. So $C$ is a bipartite graph containing the edges of $G$ that join a vertex in $A$ with a vertex in $B$.

To prove the lemma, we shall show that it is possible to find an edge-colouring of $G$ with $d(G)$ colours. In fact, we shall give a procedure for finding such an edge-colouring. The procedure is a little technical, so we shall first give an overview of the steps involved. We are not interested in efficiency, merely in the fact that the procedure can be carried out. At the start of the procedure, all the edges of $G$ are assumed to be uncoloured.

\begin{enumerate}[{Step} 1.]
\item We shall find equalized edge-colourings of $G_A$ and $G_B$ with $k$ colours, where $k = \max\{\Delta(G_A), \Delta(G_B)\} + r$. In this partial edge-colouring of $G$, we shall insist that each colour misses the same number of vertices in $A$ as it does in $B$, and that the number of vertices missed in each part is less than $2n^{2/3} + 3$.

\item We shall modify the partial edge-colouring of $G$ obtained in Step~1 by exchanging alternating paths. Once this step has been completed, each of the $k$ colour classes will be a 1-factor of $G$. During the course of Step~2, we shall colour a few of the edges of $C$, and we shall uncolour a few of the edges of $G_A$ and $G_B$ that were coloured in Step~1. We shall ensure that after Step~2 has been completed, the following three conditions hold:

\begin{enumerate}[(i)]
\item $G_A$ and $G_B$ contain the same number of uncoloured edges, and this number is less than $2n^{5/3}$.
\item If $R_A$ and $R_B$ denote the subgraphs of $G_A$ and $G_B$ respectively consisting of the uncoloured edges, both $R_A$ and $R_B$ have maximum degree less than $n^{5/6} + 1$.
\item Each vertex is incident with fewer than $3n^{5/6}$ coloured edges of $C$.
\end{enumerate}

\item We shall find equalized edge-colourings of $R_A$ and $R_B$ with exactly $j = \ceiling{n^{5/6}} + r + 1$ colours. We shall then colour some of the uncoloured edges of $C$ with these $j$ colours, so that each of the $j$ colour classes is a 1-factor of $G$. At the end of Step~3, all the edges in $G_A$ and $G_B$ will be coloured, and so will a few of the edges of $C$. Each of the $k + j$ colour classes will be a 1-factor of $G$.

\item At the start of Step~4, all of the edges that remain uncoloured belong to $C$. Also, each colour class is a 1-factor, so the subgraph of $G$ consisting of the uncoloured edges is regular, of degree $d(G) - k - j$. This subgraph is bipartite, so by Theorem~\ref{konig} (K\"onig's Theorem) we can colour its edges with $d(G)-k-j$ colours.
\end{enumerate}

At the conclusion of Step~4, all the edges of $G$ will have been coloured, with $d(G)$ colours. We shall now describe the steps in detail.

\begin{center}Step~1.\end{center}

Let $k = \max\{\Delta(G_A), \Delta(G_B)\} + r$. By Theorem \ref{vizing}, $\chi'(G_A)$ and $\chi'(G_B)$ are at most $k$, so by Theorem \ref{mcdiarmid}, we can find equalized edge-colourings of $G_A$ and $G_B$ using $k$ colours $c_1, \dots, c_k$. Note that $k > \delta(G_A) > rn / 2$.

As $G$ is regular, $G_A$ has the same number of edges as $G_B$, which we shall suppose is $m$. As the edge-colourings of $G_A$ and $G_B$ are equalized, each colour appears on either $\floor{m/k}$ or $\ceiling{m/k}$ edges. In our edge-colourings, we shall insist that each colour appears the same number of times on edges of $G_A$ as it does on edges of $G_B$. We can do this because, as $G_A$ and $G_B$ each has $m$ edges, the number of colours that appear on $\floor{m/k}$ edges of $G_A$ equals the number of colours that appear on $\floor{m/k}$ edges of $G_B$ (and similarly for the number of colours that appear on $\ceiling{m/k}$ edges).

It follows from the assumption that
\[ \max\{\Delta(G_A), \Delta(G_B)\} - \min\{\delta(G_A), \delta(G_B)\} < n^{2/3} \]
that the number of colours that miss a given vertex in $A$ is always less than $n^{2/3} + r$. So the average number of vertices in $A$ that a colour misses is less than
\[ \frac{n(n^{2/3} + r)}{k} < \frac{n(n^{2/3} + r)}{rn / 2} \le 2n^{2/3} + 2. \]
As any two colour classes differ in size by at most one, in our partial edge-colouring of $G$, each colour misses fewer than $2n^{2/3} + 3$ vertices in $A$. (And clearly the same holds for vertices in $B$.)

\begin{figure}
\begin{center} \begin{tikzpicture}
\fill[gray!10,rounded corners=5mm] (0,0) rectangle +(2,5) (4,0) rectangle +(2,5);
\draw[gray, dashed, thick] (1,1)--(5,2.5) (5,1)--(1,2.5) (1,4)--(5,4);
\draw[gray, thick] (1,2.5)--(1,4) (5,2.5)--(5,4);
\fill[black] (1,1) circle (3pt) node[left] {$a$}
(5,1) circle (3pt) node[right] {$b$}
(1,2.5) circle (3pt) node[left] {$a_1$}
(1,4) circle (3pt) node[left] {$a_2$}
(5,2.5) circle (3pt) node[right] {$b_1$}
(5,4) circle (3pt) node[right] {$b_2$};
\draw[black] (1,0) node[below] {$A$} (5,0) node[below] {$B$};
\end{tikzpicture} \end{center} \caption{The alternating path $P$. Dashed lines indicate uncoloured edges, and solid lines indicate edges coloured $c_i$.} \label{pdiagram}
\end{figure}
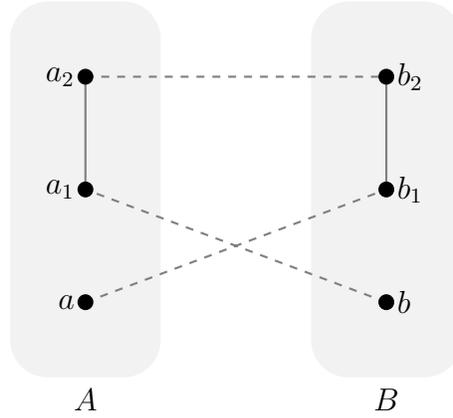

\begin{center}Step~2.\end{center}

We shall show that by exchanging alternating paths we can increase the size of the colour classes until each colour class is a 1-factor of $G$. During the course of Step~2, we shall uncolour some of the edges of $G_A$ and $G_B$, and we shall colour some of the edges of $C$. We shall denote by $R_A$ and $R_B$ the subgraphs of $G_A$ and $G_B$ consisting of their uncoloured edges. During the course of Step~2 the graphs $R_A$ and $R_B$ will change. They will initially be empty, but each time we exchange an alternating path, one edge will be added to each of $R_A$ and $R_B$.

We shall ensure that, after Step~2 has been completed, the following three conditions hold:
\begin{enumerate}[(i)]
\item $G_A$ has the same number of uncoloured edges as $G_B$, and this number is less than $2n^{5/3}$.
\item $\Delta(R_A)$ and $\Delta(R_B)$ are less than $n^{5/6} + 1$.
\item Each vertex is incident with fewer than $3n^{5/6}$ coloured edges of $C$.
\end{enumerate}

With condition (ii) in mind, we say that an edge is \emph{good} if it is not in $R_A$ or $R_B$, and both its ends have degree less than $n^{5/6}$ in $R_A$ or $R_B$. Thus we may add a good edge to $R_A$ or $R_B$ without violating condition (ii).

Our strategy is as follows. We shall consider the $k$ colours $c_1, \dots, c_k$ in turn. Each colour misses the same number of vertices in $A$ as it does in $B$. So for a given colour $c_i$, where $1 \le i \le k$, we can partition the vertices that miss $c_i$ into pairs, with one vertex from each pair belonging to $A$ and the other belonging to $B$. We shall exchange exactly one alternating path for each such pair. Suppose $(a, b)$ is one of our pairs, where $a \in A$, $b \in B$, and both vertices miss the colour $c_i$. We shall exchange an alternating path $P$ from $a$ to $b$, consisting of five edges, where the first, third and fifth edges are uncoloured and the second and fourth edges are good edges coloured $c_i$. (See Figure \ref{pdiagram}.) After $P$ is exchanged, $a$ and $b$ will be incident with edges of colour $c_i$, and one good edge will be added to each of $R_A$ and $R_B$.

Before demonstrating how such paths can be found, we shall show that at the end of Step~2, we can be sure that conditions (i), (ii) and (iii) will hold. After Step~1 has been completed, each vertex is missed by fewer than $n^{2/3} + r$ colours, so there will always be fewer than $n(n^{2/3} + r) < 2n^{5/3}$ edges in each of $R_A$ and $R_B$. Therefore at the end of Step~2, condition (i) will hold. And as we only ever add good edges to $R_A$ and $R_B$, condition (ii) will also hold.

We shall now show that condition (iii) will also hold. Let $v$ be a vertex, which, without loss of generality, we assume belongs to $A$. After Step~2 has been completed, the number of coloured edges of $C$ that are incident with $v$ will be equal to the number of alternating paths containing $v$ that have been exchanged. The number of such alternating paths of which $v$ is the first vertex will be equal to the number of colours that missed $v$ at the end of Step~1, which is less than $n^{2/3} + r$. The number of alternating paths in which $v$ is the fourth or fifth vertex will be equal to the degree of $v$ in $R_A$, and so will be less than $n^{5/6} + 1$. Hence the number of coloured edges of $C$ that are incident with $v$ will be less than
\[ (n^{2/3} + r) + (n^{5/6} + 1) < 3n^{5/6}. \]
This applies to all vertices in $G$, and so condition (iii) will be satisfied.

We shall now describe how the paths can be found. Suppose $(a, b)$ is one of our pairs, where $a \in A$, $b \in B$, and both vertices miss the colour $c_i$. Let $N_B$ be the set of vertices in $B$ that are joined with $a$ by an uncoloured edge and are incident with a good edge coloured $c_i$. Likewise, let $N_A$ be the set of vertices in $A$ that are joined with $b$ by an uncoloured edge and are incident with a good edge coloured $c_i$.

There are fewer than $2n^{5/3}$ edges in $R_B$, so there are fewer than $4n^{5/6}$ vertices of degree at least $n^{5/6}$ in $R_B$, and hence there are fewer than $8n^{5/6}$ vertices in $B$ that are incident with a non-good edge coloured $c_i$. In addition, there are fewer than $2n^{2/3} + 3$ vertices in $B$ that are missed by the colour $c_i$. So the number of vertices in $B$ that are not incident with a good edge coloured $c_i$ is less than
\[ 8n^{5/6} + 2n^{2/3} + 3 < 11n^{5/6}. \]
By symmetry, the same holds for vertices in $A$.

So for any vertex $v \in V(G)$, the number of edges that join $v$ with a vertex $w$ in the other part, where $w$ is incident with a good edge coloured $c_i$, is more than
\[ rn/2 + 14rn^{5/6} - 3n^{5/6} - 11rn^{5/6} \ge rn/2. \]
In particular, there are more than $rn/2$ edges joining $a$ with vertices in $N_B$, and more than $rn/2$ edges joining $b$ with vertices in $N_A$. And because $G$ has multiplicity at most $r$, it follows that $\size{N_A} > n/2$ and $\size{N_B} > n/2$.

Let $M_B$ be the set of vertices in $B$ that are joined with a vertex in $N_B$ by an edge of colour $c_i$, and let $M_A$ be the set of vertices in $A$ that are joined with a vertex in $N_A$ by an edge of colour $c_i$. Note that $M_B$ will have the same size as $N_B$, but some vertices may be in both (similarly with $M_A$ and $N_A$).

Suppose we choose a vertex $b_1 \in N_B$. Let $b_2 \in M_B$ be the vertex joined with $b_1$ by an edge of colour $c_i$. As each vertex in $M_B$ is joined with more than $n/2$ vertices in $A$ by uncoloured edges, and the size of $M_A$ is more than $n/2$, we must be able to find a vertex $a_2 \in M_A$ that is joined with $b_2$ by an uncoloured edge. Let $a_1 \in N_A$ be the vertex joined with $a_2$ by an edge of colour $c_i$. Then $P = a b_1 b_2 a_2 a_1 b$ is an alternating path of five edges, where the first, third and fifth edges are uncoloured and the second and fourth edges are good edges coloured $c_i$.

If we exchange $P$, the colour $c_i$ appears on edges incident with $a$ and $b$. By finding such paths for all pairs of vertices $(a,b)$ that miss $c_i$, we can increase the number of edges coloured $c_i$ until the colour class is a 1-factor of $G$. By doing this for all colours, we can make each of the $k$ colour classes a 1-factor of $G$.

\begin{center}Step~3.\end{center}

Each of the colour classes for the colours $c_1, \dots, c_k$ is now a 1-factor of $G$. We shall now consider the graphs $R_A$ and $R_B$ that consist of the uncoloured edges of $G_A$ and $G_B$ respectively. $R_A$ and $R_B$ each have fewer than $2n^{5/3}$ edges and maximum degree less than $n^{5/6} + 1$. Let $j = \ceiling{n^{5/6}} + r + 1$. By Theorems \ref{vizing} and \ref{mcdiarmid}, we can give $R_A$ and $R_B$ equalized edge-colourings with $j$ colours, $c_{k+1}, \dots, c_{k+j}$. As with the edge-colourings we found in Step~1, we shall insist that in the edge-colourings of $R_A$ and $R_B$, each colour appears on the same number of edges in $R_A$ as it does in $R_B$. We can do this because $R_A$ and $R_B$ have the same number of edges.

There are fewer than $2n^{5/3}$ edges in each of $R_A$ and $R_B$, and $j > n^{5/6}$, so each of the colours $c_{k+1}, \dots, c_{k+j}$ appears on fewer than
\[ \frac{2n^{5/3}}{j} + 1 < 3n^{5/6} \]
edges in each of $R_A$ and $R_B$. We shall now colour some of the edges of $C$ with the $j$ colours $c_{k+1}, \dots, c_{k+j}$ so that each of these colour classes becomes a 1-factor of $G$.

We shall perform the following procedure for each of the $j$ colours in turn. Given a colour $c_i$, where $k + 1 \le i \le k + j$, we let $A_i$ and $B_i$ be the sets of vertices in $A$ and $B$ respectively that are incident with edges coloured $c_i$. Note that $A_i$ and $B_i$ have the same size, and as $R_A$ and $R_B$ each contain fewer than $3n^{5/6}$ edges coloured $c_i$, $A_i$ and $B_i$ contain fewer than $6n^{5/6}$ vertices each. Let $C_i$ be the subgraph of $C$ obtained by deleting the vertex sets $A_i$ and $B_i$ and removing all coloured edges.

Each vertex in $G$ is incident with fewer than
\[ 3n^{5/6} + (\ceiling{n^{5/6}} + r) < 5n^{5/6} \]
edges of $C$ that are coloured, since fewer than $3n^{5/6}$ were coloured in Step~2 and at most $\ceiling{n^{5/6}} + r$ have been coloured already in Step~3. And each vertex has fewer than $6rn^{5/6}$ edges that join it with a vertex in $A_i$ or $B_i$. So the minimum degree of $C_i$ is more than
\[ rn/2 + 14rn^{5/6} - 6rn^{5/6} - 5n^{5/6} > rn/2, \]
and so $C_i$ has a 1-factor $F$ by Lemma \ref{hallcm}. If we colour the edges of $F$ with the colour $c_i$, then every vertex in $G$ is incident with an edge of colour $c_i$, and so the colour class is now a 1-factor of $G$.

We repeat this procedure for each of the colours $c_{k+1}, \dots, c_{k+j}$. After this has been done, each of these $j$ colour classes is a 1-factor of $G$. So at the conclusion of Step~3, all of the edges in $G_A$ and $G_B$ are coloured, some of the edges of $C$ are coloured, and each of the $k+j$ colour classes is a 1-factor of $G$.

\begin{center}Step~4.\end{center}

Let $R$ be the subgraph of $G$ consisting of the remaining uncoloured edges. These edges all belong to $C$, so $R$ is a subgraph of $C$ and hence is bipartite. As each of the $k + j$ colour classes is a 1-factor of $G$, $R$ is regular of degree $d(R) = d(G) - k - j$. Note that since
\[ k < d(G) - (rn/2 + 14rn^{5/6}) + r, \]
and $j < 2n^{5/6}$, $d(R) > rn/2$. By Theorem \ref{konig} (K\"onig's Theorem) we can colour the edges of $R$ with $d(R)$ colours $c_{k + j + 1}, \dots, c_{d(G)}$. Clearly each of these colour classes is a 1-factor of $G$.

This completes our edge-colouring of $G$ with $d(G)$ colours. Each of the colour classes is a 1-factor, so $G$ is 1-factorizable. \end{proof}

\section*{Acknowledgements}

The author would like to thank Richard Mycroft and Mark Walters for some helpful discussions.

\end{document}